\documentclass[12pt,reqno]{amsart}
\usepackage{times}
\usepackage[T1]{fontenc}
\usepackage{mathrsfs}
\usepackage{latexsym}
\usepackage[dvips]{graphics}
\usepackage{epsfig}
\usepackage{amsmath,amsfonts,amsthm,amssymb,amscd}
\input amssym.def
\input amssym.tex
\usepackage{color}
\usepackage{hyperref}
\usepackage{url}

\newcommand{\burl}[1]{\textcolor{blue}{\url{#1}}}

\newcommand{\twocase}[5]{#1 \begin{cases} #2 & \text{{\rm #3}}\\ #4
&\text{{\rm #5}} \end{cases}   }

\title{Sums and differences of correlated random sets}

\author[Do]{Thao Do}
\email{\textcolor{blue}{\href{mailto:thao.do@stonybrook.edu}{thao.do@stonybrook.edu}}}
\address{Mathematics Department, Stony Brook University, Stony Brook, NY, 11794}

\author[Kulkarni]{Archit Kulkarni}
\email{\textcolor{blue}{\href{mailto:auk@andrew.cmu.edu}{auk@andrew.cmu.edu}}}
\address{Department of Mathematical Sciences, Carnegie Mellon University, Pittsburgh, PA 15213}

\author[Miller]{Steven J. Miller}
\email{\textcolor{blue}{\href{mailto:sjm1@williams.edu}{sjm1@williams.edu}},  \textcolor{blue}{\href{Steven.Miller.MC.96@aya.yale.edu}{Steven.Miller.MC.96@aya.yale.edu}}}
\address{Department of Mathematics and Statistics, Williams College, Williamstown, MA 01267}

\author[Moon]{David Moon}
\email{\textcolor{blue}{\href{mailto:dm7@williams.edu}{dm7@williams.edu}}}

\address{Department of Mathematics \& Statistics, Williams College, Williamstown, MA
01267}

\author[Wellens]{Jake Wellens}
\email{\textcolor{blue}{\href{mailto:jwellens@caltech.edu}{jwellens@caltech.edu}}}
\address{Department of Mathematics, California Institute of Technology, Pasadena, CA 91125}

\subjclass[2010]{11B13, 11P99 (primary), 05B10, 11K99, 82B26 (secondary).}

\keywords{More Sum Than Difference sets, correlated random variables, phase transition.}

\date{\today}

\thanks{This research was conducted as part of the 2013 SMALL REU program at Williams College and was partially supported by NSF grant DMS0850577 and Williams College; the third named author was partially supported by NSF grants DMS0970067 and DMS1265673. We would like to thank our  colleagues from SMALL for helpful discussions, and to Kevin O'Bryant for suggesting a variant of this problem at CANT 2013.}

\usepackage[margin=1in]{geometry}
\usepackage{mathtools}

\newcommand\be{\begin{equation}}
\newcommand\ee{\end{equation}}
\newcommand\bea{\begin{eqnarray}}
\newcommand\eea{\end{eqnarray}}

\newtheorem{thm}{Theorem}[section]
\newtheorem{conj}[thm]{Conjecture}
\newtheorem{cor}[thm]{Corollary}

\newtheorem{lemma}[thm]{Lemma}
\newtheorem{prop}[thm]{Proposition}

\newtheorem{defi}[thm]{Definition}

\newtheorem{remark}[thm]{Remark}

\numberwithin{equation}{section}

\usepackage{color}

\usepackage[textsize=tiny]{todonotes}
\begin{document}

\maketitle
\begin{abstract}
Many fundamental questions in additive number theory (such as Goldbach's conjecture, Fermat's last theorem, and the Twin Primes conjecture) can be expressed in the language of sum and difference sets. As a typical pair of elements contributes one sum and two differences, we expect that $|A-A| > |A+A|$ for a finite set $A$. However, in 2006 Martin and O'Bryant showed that a positive proportion of subsets of $\{0, \dots, n\}$ are sum-dominant, and Zhao later showed that this proportion converges to a positive limit as $n \to \infty$. Related problems, such as constructing explicit families of sum-dominant sets, computing the value of the limiting proportion, and investigating the behavior as the probability of including a given element in $A$ to go to zero, have been analyzed extensively.

We consider many of these problems in a more general setting. Instead of just one set $A$, we study sums and differences of pairs of \emph{correlated} sets $(A,B)$. Specifically, we place each element $a \in \{0,\dots, n\}$ in $A$ with probability $p$, while $a$ goes in $B$ with probability $\rho_1$ if $a \in A$ and probability $\rho_2$ if $a \not \in A$. If $|A+B| > |(A-B) \cup (B-A)|$, we call the pair $(A,B)$ a \emph{sum-dominant $(p,\rho_1, \rho_2)$-pair}. We prove that for any fixed $\vec{\rho}=(p, \rho_1, \rho_2)$ in $(0,1)^3$, $(A,B)$ is a sum-dominant $(p,\rho_1, \rho_2)$-pair with positive probability, and show that this probability approaches a limit $P(\vec{\rho})$. Furthermore, we show that the limit function $P(\vec{\rho})$ is continuous. We also investigate what happens as $p$ decays with $n$, generalizing results of Hegarty-Miller on phase transitions. Finally, we find the smallest sizes of MSTD pairs.
\end{abstract}

\tableofcontents

\section{Introduction}\label{sec:intro}
Given a finite set $A\subset \mathbb{Z}$, it is natural to compare the sizes of its sum set $A+A$ and difference set $A-A$, which are defined as
\be A+A \ = \ \{a+b:a,b\in A\}, \ \ \ A-A\ = \ \{a-b: a,b\in A\}. \ee We have two competing influences on their respective cardinalities. For any $a\in A$, $a-a$ is always equal to $0$ while $a+a$ is different for different values of $a$. On the other hand, since addition is commutative while subtraction is not, any two different numbers  $a,b\in A$ generate two differences $a-b$ and $b-a$ but only one sum $a+b$. We thus expect that most of the time the size of the difference set is at least that of the sum set; however, this is not always the case. A set whose sum set has more elements than its difference set is called \emph{sum dominant}, or a \textit{More Sums Than Differences} (MSTD) set. One of the earliest examples is due to Conway from the 1960's: $\{0,2,3,4,7,11,12,14\}$.

We briefly review some of the key results in the field. Martin and O'Bryant \cite{MO} in 2002 proved that $p_{\rm MSTD}(1/2;n)$, the probability that a uniformly chosen random subset of \be I_n\ :=\ \{0, 1, \dots, n\}\ee is an MSTD set, is greater than a positive constant for all $n\geq 14$; note that choosing subsets uniformly is equivalent to taking each element of $I_n$ independently of the others to be in our set with probability $p$ (and hence our notation). A similar result holds if instead each element of $I_n$ is chosen independently of the others with a fixed non-zero probability $p$, and again $p_{\rm MSTD}(p;n) > 0$. This is somewhat contrary to our original intuition that MSTD sets should be rare, though we will see later that this percentage, while positive, is quite small. Subsequent work by Zhao \cite{Zh2} proved that $p_{\rm MSTD}(1/2;n)$ converges to a limit when $n\to\infty$, and Iyer, Lazarev, Miller and Zhang \cite{ILMZ} generalized these results to comparisons of linear combinations of a set. These proofs are probabilistic and non-constructive; see \cite{Na,MOS,MPR,Zh1} for explicit constructions of infinite families of MSTD sets. Other results include the work of Hegarty and Miller \cite{HM} on the behavior of $p_{\rm MSTD}(p_n;n)$ as the probability $p_n$ of including an element in $A \subset I_n$ decays with $n$, and Hegarty's \cite{He} proof that the smallest size of an MSTD set is 8 and the example found by Conway is the smallest sum dominant set up to linear transformation.

All of the literature to date has looked at sums and differences of a set with itself. In this paper, we extend the theory to combinations of two subsets of integers (see \cite{DKMMWW} for another generalization, specifically to subsets of $D$-dimensional polytopes). Given two finite sets of integers $A$ and $B$, define their sum set and difference set by \bea A+B & \ = \ & \{a+b:a\in A , b\in B\},\nonumber\\ \pm (A-B) &\  = \ & (A-B)\cup (B-A)\ = \ \{a-b,b-a: a\in A, b\in B\}. \eea We investigate sums and differences of \emph{pairs} of subsets $(A,B) \subset \{0,1, \dots, n\}$, which are selected according to the dependent random process described below.

\begin{defi}
Fix a $\vec{\rho} = (p, \rho_1, \rho_2)\in [0,1]^3$. We call $(A,B)$ a \emph{$\vec{\rho}$-correlated pair} if each element $k\in I_n$ is chosen into $A$ and $B$ by the following rule:
\be \mathbb{P}(k\in A)\ = \ p;\quad\ \ \ \ \mathbb{P}(k\in B |k\in A)\ =\ \rho_1;\quad\ \ \ \mathbb{P}(k\in B|k\notin A)\ =\ \rho_2.
\ee
\end{defi}

We say a correlated pair $(A,B)$ is a \textit{More Sums Than Differences} (MSTD) or \textit{sum dominant}  pair if the size of their sum set is bigger than that of their difference set: $|A+B|>|\pm(A-B)|$. For each $n$, let $P_n(\vec{\rho})$ denote the probability a randomly chosen $\vec{\rho}-$correlated pair $(A,B)$ is an MSTD pair.

If $(\rho_1,\rho_2)=(1,0)$ then $B=A$ and thus the problem is reduced to comparing the sizes of the sum set and the difference set of $A$ with itself; this is the $(A,A)$ case, and is the only one that has been studied extensively in literature so far. If we let $(\rho_1,\rho_2)=(0,1)$, then $B$ contains all elements that are not in $A$ and thus $B$ is the complement of $A$; we call this the $(A,A^c)$ case. If we let $\rho_1=\rho_2$, then $A$ and $B$ are chosen independently. Finally, if $\vec{\rho}=(0.5, 0.5, 0.5)$ then $P_n(\vec{\rho})$ is simply the proportion of pairs of subsets of $\{0,1,\dots, n\}$ that are MSTD. In this case, we call the MSTD correlated pair simply an MSTD pair.

In this paper, we address three questions regarding MSTD correlated pairs.

\begin{itemize}

\item[(1)] For a fixed probability vector $\vec{\rho}$, does $P_n(\vec{\rho})$ converge to a positive number as $n\to\infty$?

\item[(2)] If we let $\vec{\rho}$ decay with $n$, does $P_n(\vec{\rho})$ converge to $0$ as $n\to\infty$?

\item[(3)] What are the minimal sizes of an MSTD pair and what are the minimal MSTD pairs up to linear transformation? We say $(m,n)$ is a minimal size of an MSTD pair if for any MSTD pair $(A,B)$ not having that size, then either $|A|>m$ or $|B|>n$. It can thus happen that there is more than one minimal size.

\end{itemize}

To address the first question, we exploit the probabilistic methods of Martin and O'Bryant \cite{MO} and Zhao \cite{Zh2}. We first construct a pair that has an MSTD \emph{fringe}; these are the elements near the endpoints of $A$ and typically control whether or not the set is sum-dominant (see Definition \ref{def_fringe} for details). Next we show that almost all MSTD correlated pairs are rich, which essentially means that we have an MSTD fringe and that a large interval of middle sums are obtained; see Definition \ref{def_rich} for details. From this we are able to answer completely the first question.

\begin{thm}\label{thmConvergence}
For each vector $\vec{\rho}=(p,\rho_1,\rho_2)\in [0,1]^3$, the proportion of sum dominant $\vec{\rho}$-correlated pairs of $I_n$ converges to a limit $P(\vec{\rho})$ as $n\to\infty$. Moreover, $P(\vec{\rho})=0$ if
$p\in\{0,1\}$ or $\rho_1+\rho_2\in \{0,2\}$, and $P(\vec{\rho})$ is strictly positive otherwise.
\end{thm}

From Monte-Carlo experiments, Martin and O'Bryant \cite{MO} conjectured that the \emph{proportion} of MSTD sets, or  $P((0.5,1,0))$, is approximately $4.5 \times 10^{-4}$; Zhao \cite{Zh2} has derived algorithms supporting a limit of this size. Since we expect MSTD sets to be rare, we are interested in finding the maximum value of the function $P$. The following theorem says that this search is not completely hopeless.

\begin{thm}\label{thmContinuous}
The function $P:[0,1]^3\to [0,1]$, defined in Theorem \ref{thmConvergence}, is continuous and thus attains its maximum at some point.
\end{thm}

In Section \ref{sec:p_function} we investigate $P$ and conjecture that the maximum occurs at $(0.5, 0,1)$.

The second question for the $(A,A)$ case was first conjectured by Martin and O'Bryant \cite{MO} and solved there by Hegarty and Miller \cite{HM}. The question is interesting because if $(p,\rho_1,\rho_2)$ is fixed with $p>0$ and $0< \rho_1+\rho_2<2$, then the expected sizes of $A$ and $B$ are proportional to $n$ and it is reasonable to expect a positive probability of having MSTD correlated pairs. If instead we let either $p\to 0$ or $\rho_1+\rho_2\to 0$ or $2$, then the expected size of $A$ (if $p\to 0$) or $B$ (if $\rho_1 + \rho_2 \to 0$ or $2$) is no longer proportional to $n$ and it is unclear whether or not we should have a positive probability of MSTD correlated pairs.

The case studied in \cite{HM} is $(\rho_1,\rho_2)=(1,0)$ and $p\to 0$ as $n\to\infty$. Before stating their main results, we fix some notation. Let $\mathscr{X}$ be a real-valued random variable depending on some integer parameter $N$, and let $f(N)$ be a real-valued function. We write $\mathscr{X}\sim f(N)$ if for any $\epsilon_1$, $\epsilon_2>0$ there exists $N_{\epsilon_1,\epsilon_2}>0$ such that for all $N>N_{\epsilon_1,\epsilon_2}$,
\be
\mathbb{P}(\mathscr{X} \notin [(1-\epsilon_1)f(N),(1+\epsilon_1)f(N)])\ < \ \epsilon_2.
\ee
We also use standard big-Oh, small-oh and $\Theta$ notations. We write $f(x) = O(g(x))$ if there exist constants $x_0$ and $C$ such that for all $x \geq x_0$, $|f(x)|\leq Cg(x)$. If $f(x)=O(g(x))$ and $g(x)=O(f(x))$ we say $f(x)=\Theta(g(x))$. Finally, we write $f(x)=o(g(x))$ (or $g(x)\gg f(x)$) if $\lim_{x\to\infty} f(x)/g(x)=0$. The following theorem captures the main results in \cite{HM}.

\begin{thm}\label{thmHegartyMiller} [Hegarty-Miller \cite{HM}] For $p:\mathbb{N}\to (0,1)$ such that $p(N)=o(1)$ and $N^{-1}=o(p(N))$, let each $k \in I_N := \{0, \dots, N\}$ be independently chosen to be in $A$ with probability $p(N)$. The probability that $A \subset I_N$ is MSTD tends to 0.

Let $\mathscr{S}=|A+A|$, $\mathscr{D}=|A-A|$  and $\mathscr{S}^C=2N+1-\mathscr{S},\mathscr{D}^C=2N+1-\mathscr{D}$ be the sizes of their complements.

\begin{itemize}

\item[(i)] If $p=o(N^{-1/2}),$ then $\mathscr{D}\sim 2\mathscr{S}\sim(Np)^2$.

\item[(ii)] If $p=cN^{-1/2}$ for $c\in(0,\infty)$, then for $g(x)=2(e^{-x}-(1-x))/x$
\be\mathscr{S}\ \sim\ g\left(\frac{c^2}{2}\right)N\quad \ \ \text{and}\quad\ \  \mathscr{D}\ \sim\ g(c^2)N. \ee

\item[(iii)] If $N^{-1/2}=o(p)$ then $\mathscr{S}^c\sim 2\mathscr{D}^c\sim 4/p^2$.
\end{itemize}
\end{thm}
This theorem identifies $N^{-1/2}$ as the \emph{threshold function} where the phase transition happens. The ratio between sizes of the sum set and difference set behaves differently for $p$ with decay on opposite sides of this threshold. Below the threshold the ratio is almost surely $2+o(1)$ while above it is almost surely $1+o(1)$.

Building on their methods, we extend their results to our more general setting.

\begin{thm}\label{thmDecay} For fixed $\rho_1,\rho_2\in [0,1]$, $0<\rho_1+\rho_2<2$ and a function $p:\mathbb{N}\to (0,1)$ such that $p(N)=o(1)$ and $N^{-1}=o(p(N))$, the probability that $(A,B) \subset I_N$ is an MSTD $(p(N),\rho_1,\rho_2)$-correlated pair tends to 0.

In particular, let $\hat{p}=p^2(2\rho_1-\rho_1^2)+2p(1-p)\rho_2$ where $p = p(N)$. Let $\mathscr{S}=|A+B|$ and $\mathscr{D}=|\pm(A-B)|$ and $\mathscr{S}^C=2N+1-\mathscr{S},\mathscr{D}^C=2n-1-\mathscr{D}$ be the sizes of their complements.

\begin{itemize}
\item[(i)] If $\hat{p}=o(N^{-1})$, then $\mathscr{D}\sim 2\mathscr{S}\sim N^2\hat{p}$.

\item[(ii)] If $\hat{p}=cN^{-1}$ for some $c\in(0,\infty)$. Let $g(x)=2(e^{-x}-(1-x))/x$, then
 \be \mathscr{S}\ \sim\ g\left(\frac{c}{2}\right)N\quad\ \ \text{and}\quad\ \ \mathscr{D}\ \sim\ g(c)N.\ee

\item[(iii)] If $N^{-1}=o(\hat{p})$, then $\mathbb{E}(\mathscr{S}^c)\sim \mathbb{E}(2\mathscr{D}^c)\sim  4/\hat{p}$.
\end{itemize}
\end{thm}

Finally, we are able to answer the first part of the third question.

\begin{thm}\label{thmMinimalSize} The minimal sizes of MSTD pairs are $(3,5)$ and $(4,4)$. Examples of MSTD pairs with such sizes are
\be A\ =\ \{0,1,4,6,7\},\quad B\ =\ \{2,3,5\}\nonumber\ee
\be A\ =\ \{0,1,4,6\},\quad B\ =\ \{0,2,5,6\}.\ee
\end{thm}

We attack these three questions in their listed order. In Sections \S\ref{sec:positive percentage} and \S\ref{sec:p_function} we address the first question by proving Theorem \ref{thmConvergence} and Theorem \ref{thmContinuous}. We next investigate the decay of $p$ in \S\ref{sec:decay} and prove the result about minimal MSTD pairs in Section \S\ref{sec:smallest pair}. We conclude with a list of questions for future research.

\section{Positive percentage of MSTD correlated pairs}\label{sec:positive percentage}

In this section we generalize the arguments of \cite{MO} and \cite{Zh2} to the case of $(p, \rho_1, \rho_2)$-pairs $(A,B)$. Let $I_n := \{0, \dots, n\}$; we also write $[0, n]$ for this interval. Additionally, $n - A = \{n - a: a \in A\}$; we frequently enclose it in parentheses when performing unions or intersections to clearly identity the sets.
We first prove an easy yet very helpful result.
\begin{prop}\label{prop_0}
If $p\in\{0,1\}$ or $\rho_1+\rho_2\in\{0,2\}$ then there is no $\vec{\rho}-$correlated MSTD pair in $I_n$.
\end{prop}
\begin{proof}
It is easy to see that if $p=0$ or $1$, the set $A$ is, respectively, the empty set or $I_n$.  In the first case, $|A+B|=|A-B|=0$ for any set $B$. In the latter case, if $l$ and $s$ are the largest and smallest elements of $B$ $(0\leq s\leq l\leq n)$, then $A+B=\{s, s+1, \dots, n+l\}$ and $\pm(A-B)=\{-d, -(d-1), \dots, d-1, d\}$ where $d=\max\{n-s,l\}$. Hence $|A+B|=(n+l)-s+1=(n-s)+l+1\leq 2d+1=|\pm(A-B)|$. In either case, there is no MSTD correlated pair (for any $n$). Similarly, if $\rho_1+\rho_2\in\{0,2\}$ or equivalently $(\rho_1,\rho_2)\in\{(0,0);(1,1)\}$, $B=\emptyset$ or $I_n$, and there is no MSTD pair either.
\end{proof}

Therefore  from now on we assume $0<p<1$ and $0<\rho_1+\rho_2<2$ unless stated otherwise.

We now establish two useful lemmas which are analogous to Lemmas 7 and 11 in \cite{MO}. Their proofs follow from Bayes's formula, and for completeness are given in Appendix \ref{sec:lem_p3_p4}.

\begin{lemma}\label{lem_p3_p4} Let $(A,B)$ be a $(p,\rho_1,\rho_2)$-correlated pair. For any $k\in [0, 2n]$, the probability $k$ does not belong to the sum set $A+B$ is
\be
\twocase{\mathbb{P}(k\notin A+B) \ =\ }{\rho_3^{\min\{\frac{k+1}{2},\frac{2n-k+1}{2}\}}}{if $k$ is odd}{\rho_4\rho_3^{\min\{\frac{k}{2},\frac{2n-k}{2}\}}}{if $k$ is even,}
\ee
where \be\rho_3\ =\ (1-\rho_1)^2p^2+2(1-\rho_2)p(1-p)+(1-p)^2\ \ \ {\rm and} \ \ \ \rho_4\ =\ (1-\rho_1)p+(1-p).\ee
\end{lemma}

\begin{lemma}\label{lem_p3_difference} Let $(A,B)$ be a $(p,\rho_1,\rho_2)$-correlated pair. For any $k\in [-n, n]$,
\be \twocase{\mathbb{P}(k\notin (A-B)\cup (B-A))\ \leq\ }{\rho_3^{n/3}}{if $1\leq k\leq n/2$}{\rho_3^{n-k}}{if $n/2<k\leq n$,}\ee where $\rho_3$ is defined in Lemma \ref{lem_p3_p4}.
\end{lemma}

\begin{remark} It is easy to check that when $(\rho_1,\rho_2)=(1,0)$, $\rho_3=1-p^2$ and $\rho_4=1-p$; note this is consistent with the results in \cite{MO} and \cite{Zh2}.
\end{remark}

We next give definitions of MSTD fringe tuples and rich MSTD pairs, analogous to Definitions 2.1 and 2.4 in \cite{Zh2}. As we will see, these definitions characterize the behavior of almost all MSTD pairs in the limit.

\begin{defi}[MSTD fringe tuple]\label{def_fringe} For $k < n/2$ and subsets $L, L', R, R'$  of $[0,k]$, we say $(L, L', R, R'; k)$ is an \emph{MSTD fringe tuple} if
\be |(L + L') \cap [0,k]| + |(R+R') \cap [0,k]|\ >\ 2|((L + R') \cap [0,k]) \cup ((L'+R) \cap [0,k])|.
\ee
\end{defi}

\begin{defi}[Rich MSTD pair]\label{def_rich} We call a pair of subsets $(A, B) \subset S$ a \emph{rich MSTD pair with fringe tuple $(L, L', R, R'; k)$} if
\begin{itemize}
\item[(i)] $A \cap [0,k] = L,  \hspace{5pt} B \cap [0, k] = L'$,
\item[(ii)] $(n-A) \cap [0, k] = R, \hspace{5pt}  (n-B) \cap [0,k] = R'$,
\item[(iii)] $[k+1, 2n-k-1] \subseteq A+B$.
\end{itemize}
\noindent The smallest such $k$ is called the order of this rich pair.
\end{defi}
Any pair $(A,B)$ satisfying (i) and (ii) is said to have \emph{fringe profile given by} $(L,L',R,R';k)$. These two conditions and Definition 2.4 imply that $A+B$ has more ``extreme'' elements than $\pm (A-B)$ (here ``extreme'' refers to the smallest $k$ elements and the largest $k$ elements of $I + I$ and $I - I$). If condition (iii) is also satisfied (i.e., the pair $(A, B)$ is rich) then $A+B$ has all the ``non-extreme'' elements of $I+I$, and thus $|A+B| > |\pm (A-B)|$. This intuition is formalized in the proof of the following lemma, and justifies our nomenclature.

\begin{lemma} A rich MSTD pair is an MSTD pair.
\end{lemma}

\begin{proof} The proof is similar to the proof of Lemma 2.5 in \cite{Zh2}. We want $|A+B| > |\pm(A-B)|$. It suffices to show the following two inequalities:
\be\label{eq:inequalityone} |(A+B) \cap ([0, k] \cup [2n-k, 2n])|\ >\ |\pm(A-B) \cap ([-n, -n+k] \cup [n-k, n])|
\ee
\be\label{eq:inequalitytwo} |(A+B) \cap [k+1, 2n-k-1]|\ \geq\ |\pm(A-B) \cap[-n+k+1, n-k-1]|.
\ee

The inequality in \eqref{eq:inequalitytwo} follows immediately from the richness criterion.  To prove \eqref{eq:inequalityone}, note that
\bea (A+B) \cap [0,k] &\ =\ & (L+L') \cap [0,k] \nonumber\\
(A+B) \cap [2n-k, 2n] &=& ((n-R)+(n-R'))\cap[2n-k, 2n]\ =\ 2n-(R+R')\cap[0,k]\nonumber\\
(A-B) \cap [-n, -n+k] &=& (L- (n-R'))\cap[-n, -n+k]\ =\ (L+R')\cap[0,k] - n\nonumber\\
(B-A) \cap [-n, -n+k] &=& (L' -(n-R))\cap[-n, -n+k]\ =\ (L'+R)\cap[0,k] - n\nonumber\\
(A-B) \cap[n-k,n] &=& (L-(n-R'))\cap[n-k,n]\ =\ n- (L+R')\cap[0,k]\nonumber\\
(B-A) \cap[n-k, n] &=& (L'-(n-R))\cap[n-k,n]\ =\ n- (L'+R)\cap[0,k]. \eea
Hence
\be |\pm(A-B) \cap ([-n, -n+k] \cup [n-k, n])|\ =\ 2|((L + R') \cap [0,k]) \cup ((L'+R) \cap [0,k])|,\ee
while
\be |(A+B) \cap ([0, k] \cup [2n-k, 2n])|\ =\ |(L + L') \cap [0,k]| + |(R+R') \cap [0,k]|.\ee
The desired inequality then follows from the definition (\ref{def_fringe}) of an MSTD fringe tuple.
\end{proof}

Much like in \cite{Zh2}, we will see in the proof of Proposition \ref{prop_lim_exist} that \emph{almost all MSTD pairs are rich}.  Following \cite{Zh2} we define a partial order on fringe tuples below, which allows us to count fringe tuples without redundancy.

\begin{defi}[Partial ordering of fringe tuples]\label{defi:partialorderfringetuples} We say $(L, L', R, R'; k) > (M, M', T, T'; j)$ if $k > j$ and
\bea & & M  \ =\ L\cap[0,j], \ \ \ \ M' \ = \ L'\cap[0,j], \ \ \ \  T\ =\ R\cap[0,j], \ \ \ \ T'\ =\ R'\cap[0,j] \nonumber\\ & & \ \ \ \ \ \ \ \ \ \ \ \  [j, k] \ \subseteq \ L+L', \ \ \ \ [j, k] \ \subseteq\ R+R'.\eea
\end{defi}

The arguments in \cite{Zh2} also show that minimal fringe tuples for a given rich pair $(A,B)$ are unique, and they are minimal in the partial order of all fringe tuples. This allows us to count rich MSTD pairs by their minimal fringe tuples.

Fix any $k > 0$. For $n > 2k$, let $\mathbb{P}_n[E]$ denote the probability that, out of all $(p,\rho_1, \rho_2) = \vec{\rho}$ correlated pairs of subsets $(A,B)$ of $[0, n]$, $A$ and $B$ satisfy the conditions prescribed by the event $E$.

Let $P_n(\vec{\rho})(L,L',R,R';k)$ be the probability that the pair $(A,B) \in I_n$ is a rich MSTD $\vec{\rho}$-pair with fringe profile $(L,L',R,R';k)$; that is
$P_n(\vec{\rho})(L,L',R,R';k)$ equals \be \mathbb{P}_n[(A,B) \text{ has fringe profile } (L,L',R, R';k) \text{ and } [k+1, 2n-k-1] \ \subseteq\  A+B].\ee We write this more compactly as
\be P_n(\vec{\rho})(L,L',R,R';k)\ :=\ \mathbb{P}_n[(L,L',R, R';k), \, [k+1, 2n-k-1] \ \subseteq\  A+B].
\ee

\begin{lemma} For any fringe profile $(L,L',R,R';k)$ and any $\vec{\rho} = (p, \rho_1, \rho_2)$, the following limit exists:
\be
P(\vec{\rho})(L,L',R,R';k)\ :=\ \lim_{n \to \infty} P_n(\vec{\rho})(L,L',R,R';k).
\ee
\end{lemma}

\begin{proof} Following the example in \cite{Zh2}, we break up the event $[k+1, 2n-k-1] \not \in A+B$ into the disjoint events
\be [k+1, j-1]\ \in\ A+B, \ \ \ \ j \not \in A+B\ee
for each $k < j \leq 2n-k$. Thus
\bea & & \mathbb{P}_n\left[(L,L',R, R';k), \, [k+1, 2n-k-1] \ \subseteq\  A+B\right] \nonumber\\
& & \ =\  \mathbb{P}_n[(L, L', R, R';k)] - \sum_{j > k}^{2n-k} \mathbb{P}_n[(L, L', R, R';k), [k+1, j-1] \in A+B; \, j \not \in A+B] \nonumber\\
& & \ =\ \mathbb{P}_{2k}[(L, L', R, R';k)] - \sum_{j > k}^{2n-k} \mathbb{P}_{j+k}[(L, L', R, R';k), [k+1, j-1] \in A+B; \, j \not \in A+B],\nonumber\\ \eea
where in the final line we have replaced the $n$ subscripts with smaller ones, which we can do because these events only involve at most $2k$ (resp. $j+k$) elements, and the probabilities do not change when we allow for more middle elements to belong (or not belong) to $A$ and $B$. Thus everything except the upper limit on the sum is independent of $n$. We send $n$ to infinity and find
\begin{align} & P(\vec{\rho})(L,L',R,R';k)\ :=\ \lim_{n \to \infty} P_n(\vec{\rho})(L,L',R,R';k) \nonumber\\
&\ =\ \mathbb{P}_{2k}[(L, L', R, R';k)] - \sum_{j > k}^{\infty} \mathbb{P}_{j+k}[(L, L', R, R';k), [k+1, j-1] \in A+B; \, j \not \in A+B].
\end{align}
Since each term in the sum is non-negative and the total sum is bounded above by 1 (as the partial sums represent legitimate probabilities), the monotone convergence theorem says the sum converges, and thus the limiting probability exists.  \end{proof}

The next definition isolates our key object of study; we prove that it exists and give a formula for it in the proposition that follows.

\begin{defi}[$P(\vec{\rho})$] For $\vec{\rho} \in [0,1]^3$, set
\be P(\vec{\rho})\ :=\ \lim_{n \to \infty} \mathbb{P}_n [(A,B) \text{ is an MSTD } (p, \rho_1, \rho_2) \text{-correlated pair }].
\ee \end{defi}

\begin{prop}\label{prop_lim_exist} The limit $P(\vec{\rho})$ exists and is given by
\be\label{sumAllFringe}
 \sum_{(L,L',R,R';k)} P(\vec{\rho})(L,L',R,R';k), \ee
where the sum is taken over all minimal fringe tuples $(L,L',R,R';k)$.
\end{prop}

\begin{proof}
As assumed, $0<p<1$ and $0<\rho_1+\rho_2<2$. Fix a positive integer $K$ and let $n$ be large enough.

Suppose $(A,B)$ is an MSTD pair of $I_n$.  Let $L, L'$ be intersections of $A,B$ with $[0,K]$ and $R,R'$ be intersections of $A,B$ with $[n-K,n]$. We will prove that when $n$ gets large, $(A,B)$ is a rich MSTD pair with probability 1. Indeed, suppose $(A,B)$ is \emph{not} a rich MSTD pair of order at most $K$.
This means either $(A,B)$ is not rich, or it is rich with order greater than $K$.

In the first case, since $(L,L',R,R')$ is not an MSTD fringe, the size of difference set is not smaller than that of the sum set on the fringes. Hence there must be at least
a middle difference, i.e., a difference in $[K-n,n-K]$, be missing (otherwise $(A,B)$ cannot be sum dominant). In the second case, since $(L,L',R,R',K)$ is a fringe pair, and yet $(A,B)$ is not a rich MSTD pair of order $K$, there must be a middle sum missing, i.e., there exists some number in $[K, 2n-K]$ that is not in $A+B$. Let $E$ denote this event. We use the result from Lemma \ref{lem_p3_p4} to calculate $\mathbb{P}(E)$. Note that since $p\neq 0,1$ and $(\rho_1,\rho_2)\neq (0,0),(1,1)$, we have $0<\rho_3<1$. We find
\be
\mathbb{P}(E)\ =\
\mathbb{P}\left(\bigcup_{i=K}^{2n-K} (i\notin A+B)\right)
\ \leq\ \sum_{i=K}^{2n-K} \mathbb{P}(i\notin A+B)
\ \leq\  4\sum_{i=K/2}^{n/2} \rho_3^{i}\leq \frac{4}{1-\rho_3}\rho_3^{K/2},
\ee which goes to zero as $K \to \infty$, proving the claim for missing at least one middle sum; the proof for the probability of missing at least one middle difference proceeds similarly, using Lemma \ref{lem_p3_difference}.

We therefore have proved that when $n$ gets large, almost all MSTD pairs are rich MSTD with fringes. Therefore, by summing over all fringes as in \eqref{sumAllFringe}, we get $P(\vec{\rho})$. Note that
each term in \eqref{sumAllFringe} exists and their sum is less than 1, hence this sum converges.
\end{proof}

\begin{prop}\label{prop_lim_positive} We have $P(\vec{\rho}) > 0$ for any $\vec{\rho}$ with $0<p<1$ and $0<\rho_1+\rho_2 < 2$.
\end{prop}

\begin{proof}
As the argument is similar to one in \cite{MO}, we only sketch the proof here. Unless $\rho_1 =0$, any MSTD fringe pair $(L,R;k)$ for $(A,A)$ works as a fringe tuple $(L,L,R,R;k)$ for $(A,B)$, and occurs with fixed positive probability. One such fringe is given in \cite{MO}:  $L=\{0,2,3,7,8,9,10\}$ and $R=\{1,2,3,6,8,9,10,11\}$. By additionally imposing that $[12, 12+j] \subset A \cap B$, for sufficiently large $j$ (which depends on $\vec{\rho}$), we can ensure that $(A,B)$ is rich with positive probability. Thus $P(\vec{\rho}) \geq P(\vec{\rho})(L,L,R,R;k) > 0$.

Now we handle the case when $\rho_1=0$. Since $\rho_2 > 0$, a fringe profile for $(A,A^c)$ occurs with positive probability in this case, and the same reasoning above will hold. Thus it suffices to exhibit a single MSTD fringe profile for $(A,A^c)$. One such fringe profile is $L=R=\{1,2,3,5,7,8\}$.
\end{proof}

\begin{proof}[Proof of Theorem \ref{thmConvergence}] The proof follows immediately from Propositions \ref{prop_0}, \ref{prop_lim_exist} and \ref{prop_lim_positive}. \end{proof}

\section{The probability function $P$}\label{sec:p_function}

We now investigate the behavior of the function $P:[0,1]^3\to [0,1]$, which gives the limiting probability of selecting an MSTD $\vec{\rho}$-correlated pair $(A,B)$ from $I_n$ as $n \to \infty$. We prove that $P$ is continuous, as stated in Theorem \ref{thmContinuous}. Afterwards we compute the probability function for $n = 8$ and discuss some conjectures about the behavior of $P$.

\begin{proof}[Proof of Theorem \ref{thmContinuous}] We first prove continuity away from the zeros; i.e., at points $\vec{\rho}$ such that $P(\vec{\rho}) \neq 0$. By Proposition 2.11, we know the zeros of $P$ are exactly the set
\be Z : = \{ (p, \rho_1, \rho_2) \in [0,1]^3 : p \in \{0, 1\} \text{ or } (\rho_1 + \rho_2) \in \{0, 2\} \}, \ee
which is a closed set in $\mathbb{R}^3$. We first show that $P$ is continuous on the open set $Z^c$, and then show that as $\vec{\rho}$ approaches any point in $Z$, the value of $P(\vec{\rho})$ approaches 0, so that $P$ is continuous on $[0,1]^3$.

We first prove that for each minimal fringe profile, $P(\vec{\rho})(L,L',R,R';k)$ is a continuous function of $\vec{\rho}$ away from $Z$ (note that these functions are also zero on $Z$). We start with the definition: \bea  P(\vec{\rho})(L,L',R,R';k) & \ :=\ & \mathbb{P}_{2k}[(L, L', R, R';k)]\nonumber\\ & & \ -\ \sum_{j > k}^{\infty} \mathbb{P}_{j+k}[(L, L', R, R';k), [k+1, j-1] \in A+B; \, j \not \in A+B].\nonumber\\ \eea The first term on the right hand side is continuous, since
\be
\mathbb{P}_{2k}[(L, L', R, R';k)]\ =\ \sum_{(A,B) \text{ has fringe profile}  (L, L', R, R';k)} \mathbb{P}_{2k}[(A,B)],
\ee
and the probability of getting $(A,B)$ is just a polynomial in $p$, $\rho_1$ and $\rho_2$, so this sum is continuous. Similarly, each term in the second sum is continuous, as we can view each term as a sum over suitable pairs $(A,B)$ of the probability of picking the pair $(A,B)$, each of which is a polynomial.

Thus to show that the infinite sum itself is continuous, it suffices to bound the tails uniformly. We will see that this follows from
\be
\mathbb{P}_{j+k}[(L, L', R, R';k), [k+1, j-1] \in A+B; \, j \not \in A+B] \, \leq \, \mathbb{P}_{j+k}[j \not \in A+B] .
\ee

The probability on the right, as computed in Lemma \ref{lem_p3_p4}, has the form $\rho_3^j$ where $\rho_3$ depends on $p,\rho_1,\rho_2$. For any fixed $\vec{\rho} \not \in Z$, restrict to a closed ball about $\vec{\rho}$ that lies entirely inside $Z^c$. We can pick $\vec{\rho}_{*}$ for which $\rho_3$  attains its maximal value $q_{*} < 1$ on this closed ball. Thus the tails are bounded by the tails of a convergent geometric series with ratio $q_{*}$, so the series converges uniformly and thus $P(\vec{\rho})(L,L',R,R';k)$ is continuous on $Z^c$.

Since \be P(\vec{\rho})\ =\ \sum_{(L,L',R,R';k)} P(\vec{\rho})(L,L',R,R';k)\ee and the summands are continuous functions of $\vec{\rho}$ on $Z^c$, it suffices to show that the tail sums
\be
 \sum_{(L,L',R,R';k) \text{ with } k > m} P(\vec{\rho})(L,L',R,R';k)
 \ee
 can be made uniformly small with $m$. This argument follows along the same lines as the proof of Proposition 2.14 in \cite{Zh2}. All contributions to this tail arise from sets where $A+B$ is missing a middle sum, where in this case ``middle'' means not in the first or the last $m$ elements. To show that these events are unlikely we use the union bound and the fact that we have a convergent infinite geometric series, starting with some maximizer (over a closed ball in $Z^c$), $q_{*}$, raised to the power $m$, which goes to zero as $m \to \infty$.

Now we must show that $P(\vec{\rho})$ approaches zero as $\vec{\rho}$ approaches any point in $Z$. First we show $P(\vec{\rho})(L, L', R, R'; k) \to 0$ as the distance dist$(\vec{\rho}, Z)$ tends to $0$. Note that
\be
P(\vec{\rho})(L,L',R,R';k)\ \leq\ \mathbb{P}_{2k}(\vec{\rho})[(L, L', R, R';k)].
\ee
As the probability on the right is a continuous function of $\rho$ which is zero on $Z$, we have \be \lim_{\text{dist}(\vec{\rho}, Z) \to 0} P(\vec{\rho})(L,L',R,R';k) = 0 \ee and thus the functions $P(\vec{\rho})(L,L',R,R';k)$ are continuous on $[0,1]^3$.
Observe that if $p > 0$ and $\rho_1 + \rho_2 > 0$, but still $\vec{\rho} \in Z$, then the same argument involving missing middle sums and differences based on Lemmas 2.1 and 2.2 works to show that $P$ is continuous at $\vec{\rho}$. So we only need to show that $P(\vec{\rho}) \to 0$ as $p \to 0$ or $\rho_1 + \rho_2 \to 0$, which is true because of theorem \ref{thmDecay} (see the next section). So, we conclude that $P(\vec{\rho})$ is continuous on $[0,1]^3$. \end{proof}

The following is an immediate consequence of the continuity of $P$ and the compactness of $[0,1]^3$.

\begin{cor} The function $P$ attains a maximum value on any compact domain. In particular, $P$ attains its maximum at some point in $[0,1]^3$. Moreover, for any $(\rho_1,\rho_2)$ fixed, $P$ as a function of $p$ attains its maximum at some point $p^\ast$. Similarly, for any fixed $p$, $P$ as a function of $(\rho_1,\rho_2)$ attains maximum at some point $(\rho_1^\ast,\rho_2^\ast)$.
\end{cor}

As $P(\vec{\rho})$ is continuous on a compact set, we can conjecture where it attains its maximal values. We start by considering the function $P_n(\vec{p})$ for $n\geq 1$, which is the probability for a $(p,\rho_1,\rho_2)$ correlated pair $(A,B)$ from $I_n$ to be an MSTD set. When $n \to \infty$ this function should converge to our function $P$. We chose $n=8$ and numerically found all MSTD pairs of subsets $(A,B)\in I_8$. Letting $\mathcal{L}_8$ be the set of all such pairs, we found $|\mathcal{L}_8|=96$. For each pair $(A,B)$ found, we recorded $|A|$, $|B|$ and $|A\cap B|$. Since each element of $\{0,1,\dots,8\}$ is chosen independently, we can calculate \be P_8(p,\rho_1,\rho_2) \ = \ \sum_{(A,B)\in \mathcal{L}_8} p^{|A|}(1-p)^{9-|A|}\rho_1^{|A\cap B|}(1-\rho_1)^{|A|-|A\cap B|}\rho_2^{|B|-|A\cap B|}(1-\rho_2)^{9-|A|-|B|+|A\cap B|}.
\ee
We plotted $P_8(\vec{\rho})$ and found its maximum appears to be at $(1/2,0,1)$. Numerical explorations suggest that $P(1/2,0,1)\approx 0.03$, which is significantly larger than $P(1/2,1,0)\approx 4.5\times 10^{-4}$. These numbers, however, should be taken with a healthy degree of skepticism. These problems are computationally intense, and it is possible that the observed behavior differs for very large $n$. For a related problem with a similar numerical difficulty, see the work in \cite{DKMMWW}

%

We end with some observations and conjectures. If we fix $0<p<1$ and $\rho_1$ not too large, we observe that $P_8$ appears to be a strictly increasing function. If we could prove this, we would then know that it would attain its maximum at $\rho_2=1$. On the other hand, if we fix $0<p<1$ and $\rho_2$ not too small, $P_8$ appears to be a strictly decreasing function, and thus would attain its maximum at $\rho_1=0$. Finally, if we fix $(\rho_1,\rho_2)$, in most cases it appears that the maximum of $P_8$ happens at some point $p$ close to $1/2$. In the specific case when $(\rho_1,\rho_2)=(0,1)$, if we assume that $P_n$ is differentiable then we can easily prove that $p=1/2$ is a critical point. Indeed, let $Q_n(p)= P_n(p,0,1)$. Since in this case $B=A^c$, we find $Q_n(p)=Q_n(1-p)$. Taking the derivative of both sides yields
\be Q'_n(p) \ = \ -Q'(1-p). \ee Consequently, $Q'_n(1/2)=0$, or $p=1/2$ is a critical point of $Q_n$, and thus of $P_n$. This suggests the following conjecture.

\begin{conj}\label{conjecture1} The maximum of the function $P$ in $[0,1]^3$ occurs at $(1/2, 0, 1)$, and $P(\vec{\rho}) = P(1/2, 0, 1)\approx 0.03$.
\end{conj}

\section{When $\vec{p}$ decays with $N$}\label{sec:decay}

As this section is devoted to generalizing Hegarty-Miller's \cite{HM} work where the density depends on the length of the interval, we use $I_N := \{0, 1, \dots, N\}$ instead of $I_n$ below to be consistent with their notation. By having $\vec{p}$ decay with $N$ we expect that there will not be a positive probability of randomly choosing an MSTD correlated pair.

In Theorem \ref{thmConvergence}, we proved that $P(\vec{\rho})>0$ unless $p \in \{0,1\}$ or $\rho_1+\rho_2\in\{0,2\}$. Therefore it is reasonable to consider two types of decay: either $p\to 0$ or $1$ while $\rho_1,\rho_2$ are fixed, or $(\rho_1,\rho_2)$ converges to either $(0,0)$  or $(1,1)$ while $p$ is fixed. In this paper we restrict ourselves to the simplest case, where we fix $(\rho_1,\rho_2)$ and let $p\to 0$. We also assume $1/N = o(p(N))$ to guarantee that $\mathbb{E}[|A|]=p(N)\cdot N$ does not tend to $0$, as otherwise $A$ is close to the empty set and the problem becomes trivial. Here we write $p(N)$ to emphasize the fact that $p$ depends on $N$. Later on, we simply write $p$ without causing confusion.

In order to prove the first and second parts of Theorem \ref{thmDecay}, we use the following definition, which resembles $(2.1)$ in \cite{HM}.

\begin{defi}
For any $(p,\rho_1,\rho_2)$-correlated random pair $(A,B)$ of $I_N$ and any integer $k\ge 1$, let
\be A_k \ = \ \{\{(a_1,b_1),\dots,(a_k,b_k)\}\subset A\times B: a_1+b_1\ = \ \cdots\ = \ a_k+b_k\}.\ee
\end{defi}

Thus $A_k$ is the set of all unordered $k$-tuples of elements in $A\times B$ having the same sum. While better notation would include $B$, we choose the simpler notation $A_k$ so that the formulas below look like the corresponding ones in \cite{HM}.

Let $X_k=|A_k|$, then if $(A,B)$ is a random pair of subsets of $I_N$, $X_k$ is a non-negative integer valued random variable. We first state a useful lemma, whose proof can be found in Appendix \ref{sec:lemma_p_hat}.

\begin{lemma} \label{lem_p_hat} Fix $a, b \in I_N$. The probability that the event $a\in A, b\in B$ or $a\in B, b\in A$ happens is $\hat{p}=p^2(2\rho_1-\rho_1^2)+2p(1-p)\rho_2$ if $a\neq b$, and $p\rho_1$ if $a=b$.
\end{lemma}

\begin{prop}\label{prop_X_k}
With $\hat{p}$ defined as in Lemma \ref{lem_p_hat}, if $\hat{p}=O(N)$ then for each $k \ge 1$ we have
\be\label{E_(X_k)} \mathbb{E}[X_k]\ \ \sim\ \ \frac{2}{(k+1)!}\ \left(\frac{\hat{p}}{2}\right)^k N^{k+1}.
\ee Moreover, $X_k\sim \mathbb{E}[X_k]$ whenever $N^{-(k+1)/k}=o(\hat{p})$.
\end{prop}

\begin{proof} As much of the proof is similar to that of Lemma 2.1 of \cite{HM}, we only give a sketch and prove the different parts. There are two types of $k$-tuples: those consisting of $2k$ distinct elements of $I_N$ (type 1 tuples) and those in which one element is repeated twice in one pair and the sum of each pair is even (type 2 tuples). Let $\xi_{1,k}(N)$ and $\xi_{2,k}$ be the total numbers of $k$-tuples of those two types. As proved in \cite{HM},
\be \xi_{1,k} \ =\ \sum_{n=2k}^{2N-2k}{\min\{\lfloor \frac{n}{2}\rfloor,\lfloor\frac{2N-n}{2}\rfloor\} \choose k}\ \sim\  \frac{2}{2^k(k+1)!}N^{k+1}
\ee
and
\be\xi_{2,k}(N) \ =\ O(N^k).\ee By Lemma \ref{lem_p_hat}, the probability for each $k$-tuple of type 1 to occur is $\hat{p}^k$, and that of type 2 is $\hat{p}^{k-1}p\rho_1$. Since $X_k$ can be written as a sum of indicator variable $Y_{\alpha}$, one for each unordered $k$-tuple $\alpha$ of type 1 or 2, we have
\be \mathbb{E}[X_k] \ =\ \xi_{1,k}(N)\cdot \hat{p}^k+\xi_{2,k}(N)\cdot \hat{p}^{k-1}p\rho_1.
\ee
By the assumption $1/N=o(p)$,
\be \frac{\xi_{2,k}(N)\hat{p}^{k-1}p\rho_1}{\xi_{1,k}\hat{p}^k} \ =\ \frac{O(N^{k}p\rho_1)}{O(N^{k+1}\hat{p})} \ =\ \frac{1}{O(N[p(2-\rho_1)+2(1-p)\rho_2 / \rho_1])} \ =\ o(1).
\ee
Hence
\be \mathbb{E}[X_n]\ \sim\ \xi_{1,k}(N)\cdot \hat{p}^k\ \sim\ \frac{2}{(k+1)!}\left(\frac{\hat{p}}{2}\right)^k N^{k+1}.
\ee

To prove the strong concentration by the mean of $X_k$ whenever $N^{-(k+1)/k}=o(\hat{p})$, we use the standard moment method as in \cite{HM}. We need to show \be\Delta\ = \ o(\mathbb{E}[X_k]^2)=o(N^{2k+2}\hat{p}^{2k}),\ee where
\be \Delta \ := \ \sum_{\alpha\sim \beta}\mathbb{P}(Y_\alpha\cap Y_\beta),
\ee
the sum being over pairs of $k$-tuples which have at least one number in common. Similar to the previous part, we can prove that the main contribution to $\Delta$ comes from pairs $\{\alpha,\beta\}$ where each $k$-tuple consists of $2k$ distinct elements and has exactly one element in common. As shown in the proof of Lemma 2.1 in \cite{HM}, the number of such pairs is $O(N^{2k+1})$. For each of the $4k-1$ elements in $I_N$, the probability they are chosen to be in two $k$-tuples, each tuple containing $2k$ distinct numbers and the two tuples having exactly one common element, is $\hat{p}^{2k-2}\cdot \mathbb{P}(E)$ where $E$ denotes the event for three distinct integers $a,b,c\in I_n$ that the pairs $(a,b)$ and $(a,c)$ are each chosen in a $k$-tuple. We use the following lemma (see Appendix \ref{sec:lemma_p_hatprime} for a proof).

\begin{lemma}\label{lem_p_hat'} Notation as above,
$\hat{p}^2/p\ = \ O(\mathbb{P}(E))$.
\end{lemma}

Using the assumption $1/N=o(p)$ we get
\be \frac{\Delta}{N^{2k+2}\hat{p}^{2k}}\ = \ \frac{O(N^{2k+1})\hat{p}^{2k-2}\mathbb{P}(E)}{N^{2k+2}\hat{p}^{2k}}\ = \ \frac{1}{O(Np)}=o(1),
\ee
or \be \Delta\ =\ o(N^{2k+2}\hat{p}^{2k})\ = \ o(\mathbb{E}[X_k]^2)\ee as we wish, completing the proof.
\end{proof}

\begin{proof}[Proof Theorem \ref{thmDecay}] We proceed similarly to the proof of Theorem \ref{thmHegartyMiller} in \cite{HM}. Although in our case we consider sums and differences of two sets instead of one, once we have the results in Proposition \ref{prop_X_k}, the rest is the same as \cite{HM}. As the arguments are similar, in parts (i) and (ii) below we analyze $\mathscr{S}$ first and then $\mathscr{D}$, while in part (iii) we first study $\mathscr{S}^c$ and then $\mathscr{D}^c$.\\

\noindent \textbf{\texttt{Proof of Part (i):}} In this regime $\hat{p}=o(1/N)$. Since $\rho_1,\rho_2$ are fixed, $p^2=O(\hat{p})$ and hence $ N^{-2}=o(\hat{p})$. Thus by \eqref{E_(X_k)}, $\mathbb{E}[X_1]\sim  \frac{1}{2}\hat{p}N^2\gg 1$. Similarly $\mathbb{E}[X_2]\sim \frac{1}{12}N^3 \hat{p}^2$ if $N^{-3/2}=o(\hat{p})$ and is $O(1)$ otherwise. Since $\hat{p}=o(1/N)$, $N^3 \hat{p}^2=o(N^2\hat{p})$. Thus in both cases $\mathbb{E}[X_2]=o(\mathbb{E}[X_1])$. Similarly, $\mathbb{E}[X_k]=o(\mathbb{E}[X_1])$ for any $k\geq 2$. In other words, as $N\to\infty$ all but a vanishing portion of pairs of elements in $(A,B)$ have distinct sums. It follows that
\be \mathscr{S}\ \sim\  \mathbb{E}[X_1]\ \sim\ \frac{1}{2}\hat{p}N^2. \ee

To prove the result for $\mathscr{D}$, we define for each $k\geq 1$
\be\label{A'_k} A'_k \ :=\ \{\{(a_1,b_1), \dots, (a_k,b_k)\}\subset A\times B\cup B\times A: a_1-b_1=\cdots = a_k-b_k\neq 0\},\ee
and proceed in a completely analogous manner to the proof of $\mathscr{S}$.\\ \

\noindent \textbf{\texttt{Proof of Part (ii):}} In this regime $\hat{p}=c/N$. Thus for any $k\geq 1$, $N^{-(k+1)/k}=o(N^{-1})=o(\hat{p})$. It follows from \eqref{E_(X_k)} that \be X_k\ \sim\  \frac{2}{(k+1)!}\left(\frac{cN^{-1}}{2}\right)^kN^{k+1} \ = \ \frac{2\cdot (c/2)^k}{(k+1)!}N. \ee
Let $\mathscr{P}$ be the partition on $A_1$ from the relation \be (a_1, b_1)\ \sim\ (a_2,b_2)\ \ \ {\rm if\ and\ only\ if}\ \ \ a_1+b_1\ = \ a_2+b_2.\ee Let $\tau_i$ denote the number of parts of size $i$ for each $i>0$. Then $\mathscr{S}=\sum_{i=0}^\infty \tau_i$. As proved in \cite{HM}, \be \mathscr{S}\ \sim\ \sum_{k=1}^\infty (-1)^{k-1}X_k\ \sim\  2\left(\sum_{k=1}^\infty \frac{(-1)^{k-1}\left(\frac{c}{2}\right)^k}{(k+1)!}\right)\cdot N=g(c/2)N.
\ee The proof for the difference set again proceeds similarly, using \eqref{A'_k}.\\ \

\noindent \textbf{\texttt{Proof of Part (iii):}} We use Lemmas \ref{lem_p3_p4} and \ref{lem_p3_difference}. Note
\be \mathbb{E}[\mathscr{S}^c] \ =\ \sum_{i=0}^{2N} \mathbb{P}(i\notin A+B)\ \sim\ 4\sum_{i=0}^{\lfloor N/2 \rfloor} \rho_3^i\ \sim\ \frac{4}{1-\rho_3}.
\ee Notice that $1-\rho_3=\hat{p}$ since $\rho_3$ and $\hat{p}$ are the probabilities of two complementary events (alternatively, we can check it directly from their formulas). So $\mathbb{E}[\mathscr{S}^c]\sim 4/{\hat{p}}$. Similarly $\mathbb{E}[\mathscr{D}^c]\sim 2/\hat{p}$.
\end{proof}

\begin{remark} The phase transition happens when $\hat{p}=\Theta(N^{-1})$. If we let $(\rho_1,\rho_2)=(1,0)$ then $\hat{p}=p^2$ and our result is consistent with the result in \cite{HM} (see Theorem \ref{thmHegartyMiller}). If we let $(\rho_1,\rho_2)=(0,1)$ then $\hat{p}=2p(1-p)=\Theta(p)$. However, since $1/N=o(p)=o(\hat{p})$, the phase transition never happens. In this $(A,A^c)$ case, the size of the difference set is always almost surely double the size of the sum set, which somewhat supports our conjecture that MSTD pairs are most abundant in the $(A,A^c)$ case.
\end{remark}

\section{Minimal MSTD pairs}\label{sec:smallest pair}

In this section we prove that the minimal MSTD pair of sets has size (3,5) or (4,4).

\begin{lemma} \label{lemmaMinimal}
If $A,B\subset I_n$ is an MSTD pair, then there must exist $a_1<a_2<a_3\in A$ and $b_1<b_2<b_3\in B$ such that $a_1+b_3=a_2+b_2=a_3+b_1$.
\end{lemma}

\begin{proof}[Proof] Assume there do not exist such $a_i, b_i$. Consider
\bea I & \ = \  & \{\{(a,b),(c,d)\}\subset A\times B: a+b=c+d\}\nonumber\\  J & = & \{\{(a,b),(c,d)\}\subset A\times B: a-b=c-d\}.\eea
 Notice that $a+b=c+d$ if and only if $a-d=c-b$. Hence we have a bijection between $I$ and $J$. In particular, this implies $|I|=|J|$ as they are finite sets.

For each $s\in [0,2n]$ and $d\in [-n,n]$, define \bea X_s & \ = \ & \{(a,b)\in A\times B: a+b=s\}\nonumber\\  Y_d & = & \{(a,b)\in A\times B: a-b=d\}.\eea
It is easy to see that \be \sum_s |X_s|\ =\ \sum_d |Y_d|\ =\ |A|\cdot |B| \ee and \be |I| \ = \ \sum_{s:|X_s|\geq 2} {|X_s| \choose 2};\quad\ \ \ |J|\ =\ \sum_{d:|Y_d|\geq 2}{|Y_d| \choose 2}.\ee We therefore find

\begin{eqnarray} |\pm(A-B)| &\ \geq \ & \ |A-B| \ = \ \sum_{d\in A-B}1\ =\ \sum_{d\in A-B} [|Y_d|-(|Y_d|-1)]\nonumber\\
&\geq &  \sum_{d\in A-B} |Y_d|-\sum_{d:|Y_d|\geq 2}{|Y_d| \choose 2}\ =\ |A|\cdot |B|-|J|.
\end{eqnarray}
Similarly
\bea |A+B| & \ =\  & \sum_{s\in A+B} 1\ =\ \sum_{s\in A+B} [|X_s|-(|X_s|-1)] \nonumber\\ & \ =\  & \sum_{s\in A+B} |X_s|-\sum_{s:|X_s|\geq 2}{|X_s| \choose 2}\ =\ |A|\cdot |B|-|I|.
\eea
The equality $|X_s|-1={|X_s| \choose 2}$ holds because $|X_s|\leq 2$ for all $s$ by our assumption that there do not exist three pairs of the same sum. Hence $|\pm(A-B)|\geq |A||B|-|J|=|A||B|-|I|=|A+B|$, contradicting the assumption that $(A,B)$ is an MSTD pair.\end{proof}

The intuition behind this lemma is that if there do not exist such $a_i,b_i$, since $a+b=c+d$ if and only if $a-d=c-b$, each \emph{collapsed} sum generates one \emph{collapsed} difference and thus the sum set cannot win.
Incidentally, this connects our two observations in the introduction: the property that the difference of any number with itself is equal to 0 is equivalent with the commutativity of addition because $a-a=b-b (=0)$ implies $a+b=b+a$ for any $a,b\in A$. The difference set has the advantage because $0$ is a big \emph{collapsed} difference.
To see this explicitly, we write
\be
|A+A|\ = \ |A|^2-|I|+\sum\left[ {|X_s| \choose 2}-(|X_s|-1)\right]\ = \ M+\sum \frac{(|X_s|-1)(|X_s|-2)}{2}\nonumber
\ee
\be
|A-A|\ = \ |A|^2-|J|+\sum\left[ {|Y_d| \choose 2}-(|Y_d|-1)\right]\ = \ M+\sum\frac{(|Y_d|-1)(|Y_d|-2)}{2},
\ee
where $M=|A|^2-|I|=|A|^2-|J|$. This implies the larger the sizes of $\{X_s\}_{s\in A+B}$ (or $\{Y_d\}_{d\in A-A}$) are, the larger the size of $A+A$ (or $A-A)$ is.
Hence $Y_0=|A|$, the biggest size a $Y_d$ or $X_s$ can obtain, will give the difference set a huge advantage. This argument also somewhat supports our conjecture that $(A,A^c)$ MSTD pairs are most abundant, because $0$ is no longer a big collapsed difference.

This purely combinatorial observation  can be applied to find some necessary conditions for a set, or a pair of sets to be sum-dominant in any setting (numbers, points in a plane, MSTD sets in two or higher dimension and so on). For example, an MSTD set of $I_n$ must not have only two elements because if so $|X_s|\leq 2$ and hence $|A+A|=M\leq |A-A|$. Likewise, if $A=\{a,b,c\}$ where $0\leq a<b<c\leq n$ is MSTD, then one of $X_s$ must be 3, which means $a+c=b+b=c+a=k$ for some integer $k$. This forces $A$ to be a symmetric set, and therefore not sum-dominant (see \cite{MO}).

Going back to the proof of theorem \ref{thmMinimalSize}, from Lemma \ref{lemmaMinimal} we immediately obtain the following corollary, as we saw above $A$ must have at least three elements.

\begin{cor} \label{cor_2_k}
There does not exist an MSTD pair $(A,B)$ of size $(2,k)$ or $(k,2)$ for any $k\geq 2$.
\end{cor}

Theorem \ref{thmMinimalSize} follows directly from the above corollary and the two following propositions.

\begin{prop}\label{prop_3_3}
There does not exist MSTD pair $(A,B)$ of size $(3,3)$.
\end{prop}

\begin{proof} Our starting point is Lemma \ref{lemmaMinimal}, which gives the existence of a triple in $A$ and a triple in $B$; as each of these sets has cardinality 3, we see these sets equal these special triples. Thus, if such an MSTD pair existed, we would have $A=\{a_1,a_2,a_3\}$ and $B=\{b_1,b_2,b_3\}$, with $|A+B|>|(A-B)\cup (B-A)|$, $a_1<a_2<a_3$ and $b_1<b_2<b_3$. Lemma  \ref{lemmaMinimal} then implies $a_1+b_3=a_2+b_2=a_3+b_1$, which gives $|A+B|\leq 9-2=7$ because we have at least two \textit{collapsed} sums. Without loss of generality we may assume $a_1\leq b_1$ and $a_1=0$.\\ \

\noindent \texttt{Case 1: $b_1=a_1$:} As $b_1 = a_1$ we have $a_3=b_3$. If $a_2=b_2$ then $A=B$. This cannot be sum-dominant because the smallest sum-dominant set has size 8. So $a_2\neq b_2$, and there are at least 3 positive differences $a_2, b_2, a_3$ in $(A-B)\cup (B-A)$. Since $0\in A-B$, $|(A-B)\cup (B-A)|\geq 7\geq |A+B|$, a contradiction.\\ \

\noindent \texttt{Case 2: $b_1>a_1$:} In this case $b_1<b_2<b_3$ are 3 positive distinct numbers in $B-A$. Thus $|(A-B)\cup(B-A)|\geq 6$. Since $|A+B|\leq 7$ we must have $(A-B)\cup(B-A)|=\{\pm b_1,\pm b_2,\pm b_3\}$. As $-b_3<b_1-a_3<b_1-a_2<b_1$, it must happen that $b_1-a_3=-b_2$ and $b_1-a_2=-b_1$, or $a_2=2b_1$ and $a_3=2b_1+b_2$. The difference $b_2-a_2=b_2-2b_1$ is bigger than $-b_1$ but less than $b_2$, and the only number in $\pm (A-B )$ between those two numbers is $b_1$, hence $b_2-2b_1=b_1$, or $b_2=3b_1$. Letting $b=b_1$, we can rewrite the pair $(A, B)$ as $A=\{0,2b,4b\}$ and $B=\{b,3b,5b\}$. It is easy to check that this is not an MSTD pair.
\end{proof}

\begin{prop}\label{prop_3_4}
There does not exist an MSTD pair $(A,B)$ of size $(3,4)$.
\end{prop}

The proof of this proposition is similar to that of Proposition \ref{prop_3_3}, except there are many more cases. Details can be found in Appendix \ref{sec:prop_3_4}. This completes the proof of Theorem \ref{thmMinimalSize}. \hfill $\Box$

\section{Conclusion and future work}

We extended the results of \cite{He,HM,MO,Zh2} of MSTD sets to MSTD correlated pairs. In particular, we proved that for each $\vec{\rho}=(p,\rho_1,\rho_2)\in [0,1]^3$ the limiting probability $P(\vec{\rho})$ of picking an MSTD $\vec{\rho}$-correlated pair exists and is positive unless $p\in\{0,1\}$ or $\rho_1+\rho_2\in\{0,2\}$. Furthermore, the function $P(\vec{\rho})$ is continuous and thus attains its maximum at some point, which we conjecture is $(1/2,0,1)$. We characterize the phase transition when we let $\vec{\rho}$ decay with $n$. Finally, we found the minimal size of an MSTD pair $(A,B)$.

We end with some of the more interesting and important open questions.

\begin{itemize}
\item[(1)] Prove or disprove Conjecture \ref{conjecture1}.

\item[(2)] Find an efficient algorithm to calculate values of $P(\vec{\rho})$, and investigate further the analytic properties of $P$.

\item[(3)] Prove the strong concentration of $\mathscr{S}^c$ and $\mathscr{D}^c$ in the case of slow decay (i.e., when $N^{-1/2}=o(\hat{p})$). Do similar results hold for other types of decay, namely $p\to 1$ or $(\rho_1,\rho_2)\to (0,0),(1,1)$?

\item[(4)] Are the examples of the MSTD pairs of size $(4,4)$ and $(3,5)$ found in  Theorem \ref{thmMinimalSize} unique up to linear transformation?

\item[(5)] Generalize the results from \cite{ILMZ} to linear combinations of correlated sets.

 \end{itemize}

\appendix

\section{Proof of Lemmas \ref{lem_p3_p4} and \ref{lem_p3_difference}}\label{sec:lem_p3_p4}

\begin{proof}[Proof of Lemma \ref{lem_p3_p4}] Let $E_{a,b}$ denote the event ($a\in A$ and $b\in B$) or ($a\in B$ and $b\in A$). For each $k\in I_n$, $k$ is not in $A+B$ if and only if for every pair $(a,b)$ in $[0,n]$ with $k=a+b$, the event $E_{a,b}$ does not happen.

If $a\neq b$ then by Bayes' formula
\begin{eqnarray}
 \mathbb{P}(E^c) & \ = \ & \mathbb{P}(E^c|a\in A, b\in A)\mathbb{P}(a\in A,b\in A)+\mathbb{P}(E^c|a\in A, b\notin A)\mathbb{P}(a\in A,b\notin A)\nonumber\\
 & & \ +\ \mathbb{P}(E^c|a\notin A, b\in A)\mathbb{P}(a\notin A,b\in A)+
 \mathbb{P}(E^c|a\notin A, b\notin A)\mathbb{P}(a\in A,b\notin A)
 \nonumber\\
&=& (1-\rho_1)^2p^2+2(1-\rho_2)p(1-p)+(1-p)^2=\rho_3.
\end{eqnarray}

If $a=b$, then similarly we find
\be \mathbb{P}(E^c)\ =\ \mathbb{P}(E^c|a\in A)\mathbb{P}(a\in A)+P(E^c|a\notin A)P(a\notin A)=(1-\rho_1)p+(1-p)=\rho_4. \ee

The claims now follow by counting how many ways $k$ can be written as sum of two elements in $I_n$ (these ways are $0+k$, $1+(k-1)$, and so on,
and the fact that no element is repeated in two different pairs (because if $a+b=a+c=k$ then $b=c$). \end{proof}

\begin{proof}[Proof of Lemma \ref{lem_p3_difference}]
We write $k$ as differences of two elements in $I_n$: $k=k-0=(k+1)-1=\cdots$. If $k>n/2$, no element is repeated in two pairs, thus similar to Lemma \ref{lem_p3_p4} we have $\mathbb{P}(k\notin \pm(A-B))
=\rho_3^{n-k}$.

If $k\leq n/2$, we use the same method used in Lemma 10 of \cite{MO}. Define the set
\be J \:=\ \left\{j: 0<j<n-k ;\quad  \left\lfloor \frac{j}{k}\right\rfloor \text{ is even}\right\}.
\ee
In other words, $J$ contains the first $k$ integers starting at $a$, then omits the next $k$ integers, and so on. It is easy to see that $|J|\geq n/3$ and $j+k\notin J$ if $j\in J$. Therefore, if we write $k=a_i-b_i$ for $b_i\in J$, we are guaranteed that the $a_i$ and $b_i$ are all distinct. We then have the same independence as before, hence
\be \mathbb{P}(k\notin \pm (A-B))\ \leq\ \mathbb{P}(\cup_{a_i-b_i=k,b_i\in J}(a_i,b_i)\notin (A\times B)\cup (B\times A)) \ = \ \rho_3^{|J|} \leq \rho_3^{n/3}.
\ee
\end{proof}

\section{Proof of Lemma \ref{lem_p_hat}}\label{sec:lemma_p_hat}

\begin{proof}[Proof of Lemma \ref{lem_p_hat}]
Denote the event in the lemma by $E$. We break the analysis into two cases, depending on whether or not $a$ equals $b$.\\ \

\noindent \texttt{Case I: $a \neq b$:} We apply Bayes' formula to $E$. Our partition is the four disjoint events on whether or not $a$ or $b$ is in $A$.
\begin{eqnarray}
\mathbb{P}(E)  & \ = \ & \mathbb{P}(E| a\in A,b\in A) \cdot\mathbb{P}(a\in A,b\in A)+  \mathbb{P}(E| a\in A,b\notin A)\cdot \mathbb{P}(a\in A,b\notin A)\nonumber\\ & &\ +\ \mathbb{P}(E| a\notin A,b\in A) \cdot\mathbb{P}(a\notin A,b\in A)+ \mathbb{P}(E| a\notin A,b\notin A)\cdot \mathbb{P}(a\notin A,b\notin A)\nonumber\\
& \ = \ & (1-(1-\rho_1)^2)\cdot p^2+\rho_2 \cdot p(1-p)+\rho_2 \cdot p(1-p)+0\nonumber\\
& \ = \ & p^2(2\rho_1-\rho_1^2)+2p(1-p)\rho_2.
\end{eqnarray}

\ \\

\noindent \texttt{Case II: $a = b$:}  We proceed similarly, and find
\begin{align}
\mathbb{P}(E)\ = \ \mathbb{P}(E|a\in A)\cdot\mathbb{P}(a\in A)+\mathbb{P}(E|a\notin A)\cdot\mathbb{P}(a\notin A)
\ = \  \rho_1 \cdot p.
\end{align}
\end{proof}


\section{Proof of Lemma \ref{lem_p_hat'}}\label{sec:lemma_p_hatprime}

\begin{proof}[Proof of Lemma \ref{lem_p_hat'}] Let $E$ be the event from the lemma, and consider the events $E_1=(a\in A,b\in B)$ and $(a\in B,b\in A)$, and $E_2= (a\in A,c\in B)$ and $(a\in B, c\in A)$. It immediately follows that $E=E_1\cap E_2$. We again use Bayes' formula, with our partition the four distinct events arising from whether or not $a$ and $b$ are in $A$ and $B$. We find
 \begin{eqnarray}
 \mathbb{P}(E) & \ = \ & \mathbb{P}(E|a\in A,a\in B)\cdot \mathbb{P}(a\in A,a\in B)+  \mathbb{P}(E|a\in A,a\notin B)\cdot \mathbb{P}(a\in A,a\notin B)\nonumber\\
 & &\ +\
 \mathbb{P}(E|a\notin A,a\in B)\cdot \mathbb{P}(a\notin A,a\in B)+
  \mathbb{P}(E|a\notin A,a\notin B)\cdot \mathbb{P}(a\notin A,a\notin B)\nonumber\\
 & \  = \ &  [p^2+2p(1-p)\rho_2+(1-p)^2\rho_2^2]\cdot p\rho_1
  \nonumber\\
  & &\ +\ [p^2\rho_1^2+2p(1-p)\rho_1\rho_2+(1-p)^2\rho_2^2]\cdot p(1-\rho_1)
   + p^2\cdot (1-p)\rho_2+0\nonumber\\
 & \ = \ & p(1-p)^2\rho_2^2+2p^2(1-p)\rho_1\rho_2(2-\rho_1)+p^3\rho_1(1+\rho_1-\rho_1^2)+p^2(1-p)\rho_2.
\end{eqnarray}

Note that we also use Bayes' formula to calculate $\mathbb{P}(E|a\in A,a\in B)$ and so on by dividing into four cases depending on whether or not each $b,c$ is in $A$ or not. Thus
\begin{eqnarray} \hat{p}^2 &\ = \ & \left[p^2\rho_1(2-\rho_1)+2p(1-p)\rho_2\right]^2\nonumber\\
&=& p^4\rho_1^2(2-\rho_1)^2+4p^3(1-p)\rho_1\rho_2(2-\rho_1)+ 4p^2(1-p)^2\rho_2^2.
\end{eqnarray}
Since $p\to 0$ and $\rho_1,\rho_2$ are fixed, both $p\mathbb{P}(E)$ and $\hat{p}^2$ have form $Ap^2+o(p^2)$ for some $A>0$; hence $\hat{p}^2=O(p\mathbb{P}(E))$ as desired.
\end{proof}

\section{Proof of Proposition \ref{prop_3_4}}\label{sec:prop_3_4}

\begin{proof}[Proof of Proposition \ref{prop_3_4}]

Assume $A=\{a_1,a_2,a_3\}$ and $B=\{b_1,b_2,b_3,b_4\}$ be an MSTD pair in $I_n$ where $0\leq a_1<a_2<a_3\leq n$ and $0\leq b_1<b_2<b_3<b_4\leq n$.
\begin{lemma}\label{lem_appendixE}
We have $d\in A-B$ if and only if $-d\in A-B$.
\end{lemma}
\begin{proof}
By Lemma \ref{lemmaMinimal}, there must exist a number $s$ such that $|X_s|=3$, or $a_1+b_i=a_2+b_j=a_3+b_k=s$ for some $1\leq k<j<i\leq 4$. There are four possibilities for $(k,j,i)$, which are $(1,2,3),(1,2,4),(1,3,4)$ and $(2,3,4)$.

It is easy to see that there is no $t$ such that $|X_t|\geq 4$. If there exists another number $s'\neq s$ such that $|X_s|=|X_{s'}|=3$, equivalently there exists  $(i',j',k')$ such that
$a_1+b_{i'}=a_2+b_{j'}=a_3+b_{k'}=s'$. Since $s\neq s'$, $i\neq i'$, $j\neq j'$ and $k\neq k'$. The only possibility is $(k,j,i)=(1,2,3)$ and $(k',j',i')=(2,3,4)$ or vice versa. In either case,
\be \label{equality1} a_1+b_3=a_2+b_2=a_3+b_1
\ee
\be a_1+b_4=a_2+b_3=a_3+b_2.
\ee
Subtracting those two chains of equalities gives $b_4-b_3=b_3-b_2=b_2-b_1$; let this common difference be $d$. From \eqref{equality1}, $a_2-a_1=b_3-b_2=d$ and $a_3-a_2=b_2-b_1=d$, which means $(a_i)$ and $(b_i)$ are two arithmetic sequences with same distance. Itt is easy to check that in this case $(A,B)$ is not an MSTD pair.

This implies there exists exactly one $s\in A+B$ such that $|X_s|=3$. From the proof of Lemma \ref{lemmaMinimal}, we see that in order for $|A+B|>|\pm (A-B)|$, it must happen $|Y_d|\leq 2$ for all $d\in A-B$, and $|\pm (A-B)|=|A-B|$, which means if $d\in A-B$, so is $-d$ and vice versa.
\end{proof}

From Lemma \ref{lem_appendixE}, we see that the smallest and largest numbers in $A-B$, which are $a_1-b_4$
and $a_3-b_1$ respectively,  must be inverse of each other. So
\be\label{eq_appendixE}
a_3-b_1=b_4-a_1
\ee

\noindent \texttt{Case 1: $a_1+b_4\neq a_3+b_1$} : so $(k,j,i)=(1,2,3)$ or $(2,3,4)$. It is easy to see that if $(A,B)$ is an MSTD pair, so is $(n-A,n-B)$ where $n-X=\{n-x: x\in X\}$. Therefore without loss of generality we can assume $(k,j,i)=(2,3,4)$, or $a_1+b_4=a_2+b_3=a_3+b_2$. Since we can translate the set by a number, assume $b_1=0$ (now $a_i, b_i$ are not necessary in $I_n$). From \eqref{eq_appendixE}, $a_1=b_4-a_3=b_2-a_1$, or $b_2=2a_1$. As $b_1<b_2, 0<2a_1$, or $a_1>0$. We can rewrite $b_i$ by $a_i$ as followed: $b_1=0; b_2=2a_1; b_4=a_3-b_1+a_1=a_1+a_3; b_2=a_1+b_4-a_2=2a_1+a_3-a_2$. So
\be A=\{a_1,a_2,a_3\};\quad  B=\{0,2a_1,2a_1+a_3-a_2,a_3\}.\ee

 We can now write down all elements (might be repeated) of $A-B$ which are $\{\pm a_1,\pm a_3, a_2, a_2-a_1-a_3, a_2-2a_1,2a_2-2a_1-a_3,a_3-2a_1\}$. By Lemma \ref{lem_appendixE}, $a_2\in A-B\Rightarrow -a_2\in A-B$, thus
one of 4 numbers $\{a_2-a_1-a_3, a_2-2a_1,2a_2-2a_1-a_3,a_3-2a_1\}$ must be equal to $-a_2$.
\\
\\
\texttt{Case 1.1: $a_2-2a_1=-a_2$} or $a_1=a_2$, a contradiction.\\
\texttt{Case 1.2: $a_3-2a_1=-a_2$}, or $a_3=2a_1-a_2<a_1$, a contradiction.\\
\texttt{Case 1.3: $a_2-a_1-a_3=-a_2$} or $a_1+a_3=2a_2$. Let $a_2-a_1=a_3-a_2=d$, then $A=\{a_1,a_1+d,a_1+2d\}$ and $B=\{0,2a_1,2a_1+d,2a_1+2d$. We can directly check that this pair is not sum-dominant.\\
\texttt{Case 1.4: $2a_2-2a_1-a_3=-a_2$}, or $2a_1+a_3=3a_2$. Let $a_2-a_1=d$, then $a_3-a_2=2a_2-2a_1=2d$. Then $A=\{a_1,a_1+d,a_1+3d\}$ and $B=\{0,2a_1,2a_1+2d,2a_1+3d$. Again it is straightforward to check that this pair is not MSTD.
\\
\\
\noindent \texttt{Case 2: $a_1+b_4=a_3+b_1$}: two pairs $(a_1,b_4)$ and $(a_3,b_1)$ have same sums and differences, hence $a_1=b_1$ and $a_3=b_4$. Without loss of generality, assume $a_1=b_1=0$ (as we can translate everything by $-a_1$) and $a_2+b_2=a_3$.
Rewrite
\be A\ = \ \{0,a_2,a_3\},\ \ \ B\ = \ \{0,a_3-a_2,b_3,a_3\}.\ee
$A-B$ consists of at most 9 elements $\{0,a_2,\pm a_3,a_2-a_3,2a_2-a_3,-b_3,a_2-b_3,a_3-b_3\}$. By Lemma \ref{lem_appendixE}, $-b_3\in A-B\Rightarrow -b_3\in A-B$. Since $0<b_3<a_3$, one of $\{a_2,2a_2-a_3,a_2-b_3,a_3-b_3\}$ must be equal to $b_3$.
\\
\\
\texttt{Case 2.1: $a_2=b_3$.}
\\
\texttt{Case 2.2: $2a_2-a_3=b_3$.}
\\
\texttt{Case 2.3: $a_2-b_3=b_3$.}
\\
\texttt{Case 2.4: $a_3-b_3=b_3$.}
\\
\\
In the first case, $|Y_0|=3$ because $0=a_1-b_1=a_2-b_3=a_3-b_4$, which contradicts our observation before that $|Y_d|\leq 2$ for all $d\in A-B$. In any of the other three latter cases, we reduce our sets to two variables $a_2$ and $a_3$. Continuing our argument based on Lemma \ref{lem_appendixE}, we can find a relation between $a_2$ and $a_3$ and check again to see that there is no such MSTD pair. This completes the proof of Proposition \ref{prop_3_4}.
\end{proof}


\ \\

\end{document}